\documentclass[12pt,reqno]{amsart}
\usepackage[T1]{fontenc}
\usepackage[utf8]{inputenc}
\usepackage{amsmath,amssymb,amsthm,mathrsfs}
\usepackage{graphicx,cite,times,xcolor,booktabs}
\definecolor{BLUE}{rgb}{0,0,1}
\usepackage[colorlinks,linkcolor=black,citecolor=black,urlcolor=black]{hyperref}
\newtheorem{theo}{Theorem}[section]
\newtheorem{coro}[theo]{Corollary}
\newtheorem{lemm}[theo]{Lemma}
\newtheorem{prop}[theo]{Proposition}
\theoremstyle{remark}

\theoremstyle{definition}

\numberwithin{equation}{section}

\title[Inverse Problem of Biharmonic Wave Equations]{Lipschitz Stability for an Inverse Problem of Biharmonic Wave Equations with Damping}
\author{  Minghui Bi  }
\address{School of Mathematics and Statistics, Center for Mathematics and
	Interdisciplinary Sciences, Northeast Normal University, Changchun, Jilin 130024, P.R.China. }
\email{bimh801@nenu.edu.cn}

\author{Yixian Gao}
\address{School of Mathematics and Statistics, Center for Mathematics and
	Interdisciplinary Sciences, Northeast Normal University, Changchun, Jilin 130024, P.R.China. }
\email{gaoyx643@nenu.edu.cn}

\thanks{ The research of Y. Gao was supported by NSFC grant
	12371187 and STDP Project of Jilin Province
	20240101006JJ}

\subjclass[2020]{35R30, 35L35, 35B35}

\keywords{biharmonic wave equation, parameter identification, inverse problem, Lipschitz stability}

\begin{document}
\begin{abstract}

We study the recovery of density and initial displacement for a clamped
biharmonic wave equation with known damping. For prescribed smooth
finite-dimensional coefficient and initial-state profiles, we prove
strong differentiability of the boundary observation map and establish
local Lipschitz stability under an injective sensitivity condition.
The proof combines natural-energy estimates with higher-regularity
bounds in the generator graph norm. For a homogeneous density and a known
nonzero initial mode with unknown amplitude, we prove injectivity and
a global Lipschitz estimate on compact parameter rectangles, using a
single curvature trace. The constants are uniform on fixed bounded damping intervals under
the stated conditions. Independent repeated-noise experiments are reported for the proved
one-mode class and for an additional finite-bandwidth square-plate
configuration, with refinement checks and reproducibility files.
\end{abstract}
\maketitle

\section{Introduction}\label{sect1}
\subsection{The inverse problem and admissible parameters}
Let $\Omega\subset\mathbb R^n$ be a bounded domain with $C^\infty$
boundary. We consider
\begin{equation}\label{a}
\begin{cases}
\rho(x)u_{tt}+\Delta^2u+\gamma u_t=0,&(0,T)\times\Omega,\\
u(0)=f,\quad u_t(0)=g,&\Omega,\\
u=\partial_\nu u=0,&(0,T)\times\partial\Omega.
\end{cases}
\end{equation}
Here $T>0$ and the damping coefficient $\gamma\ge0$ are known. The
symbol $\nu$ denotes the outward unit normal. Put
$\Sigma_T=(0,T)\times\partial\Omega$ and
\[
\mathcal Bu=(\Delta u|_{\partial\Omega},\partial_\nu\Delta u|_{\partial\Omega}),
\qquad \mathcal Y=L^2(\Sigma_T)\times L^2(\Sigma_T).
\]
The unknowns are the density $\rho$ and initial displacement $f$; the
initial velocity $g$ is fixed. All parameters, profiles, and data norms
below use a fixed normalization.
Unless indicated otherwise, spatial norms are taken over $\Omega$,
and $|\cdot|$ denotes the Euclidean norm of a parameter vector.
The space $\mathcal Y$ carries the usual product $L^2$ norm, with
time measure $dt$ and surface measure $dS$ on $\partial\Omega$
(counting measure at the endpoints when $n=1$).

We prescribe a finite set of spatial profiles and distinguish the
regularity needed to differentiate observations from the identifiability
needed to invert them. The following assumptions define the class used in the
local stability theorem.

\begin{enumerate}
\item Let $L=\Delta^2$ denote the clamped biharmonic operator, with
$D(L)=H^4(\Omega)\cap H_0^2(\Omega)$; its variational definition is
given in Section~\ref{sect2}. Choose $q\ge1$ and an
$L^2(\Omega)$-orthonormal collection of real eigenfunctions
$\phi_1,\ldots,\phi_q$ with $L\phi_k=\lambda_k\phi_k$ and
$\lambda_k>0$. Fix coefficients $c_k\in\mathbb R$ and set
$g=\sum_{k=1}^q c_k\phi_k$.
\item Fix $p\ge1$ and functions $r_j\in C_c^\infty(\Omega)$,
$1\le j\le p-1$. For $\theta=(a_0,\ldots,a_{p-1},b_1,\ldots,b_q)$
in an open set $U\subset\mathbb R^{p+q}$, define
\begin{equation}\label{eq:parameterization}
\rho_\theta=a_0+\sum_{j=1}^{p-1}a_jr_j,
\qquad f_\theta=\sum_{k=1}^q b_k\phi_k.
\end{equation}
The sum for the density is empty when $p=1$.
\item The set $U$ is chosen so that, uniformly in $\theta\in U$,
\begin{equation}\label{eq:density-bounds}
0<\rho_{\min}\le\rho_\theta(x)\le\rho_{\max}<\infty.
\end{equation}
Uniform constants for coefficient derivatives are taken on compact
subsets of $U$.
\end{enumerate}

The density is constant in a fixed collar of the boundary. The
eigenfunction initial profiles and this collar condition ensure the
compatibility needed for parameter differentiation. They are concrete
regularity assumptions, not an assumption that the inverse is stable.
They also imply $\Delta^2f_\theta|_{\partial\Omega}=0$ in the classical
sense. The direct problem itself needs much weaker assumptions, as
shown in Section~\ref{sect2}.

Write $u_{\theta,\gamma}$ for the solution and define
\[
\mathcal F_\gamma(\theta)=\mathcal Bu_{\theta,\gamma}\in\mathcal Y.
\]
Section~\ref{sec3} proves that this map is continuously differentiable.
At a reference parameter $\theta_*$, let $e_i$, $1\le i\le p+q$,
be the standard coordinate vectors of $\mathbb R^{p+q}$. Solve
\eqref{eq:sensitivity} at $(\theta_*,\gamma)$ for these directions and set
\begin{equation}\label{eq:gram-intro}
G_\gamma(\theta_*)_{ij}
=\operatorname{Re}\langle\mathcal Bz_{e_j},\mathcal Bz_{e_i}\rangle_{\mathcal Y}.
\end{equation}
Here $1\le i,j\le p+q$; this is the real symmetric Gram matrix of
the boundary sensitivities. We write $\lambda_{\min}$ for the smallest
eigenvalue of a real symmetric matrix.
The remaining identifiability assumption is that this finite matrix
is positive definite. It concerns the joint density and initial-state
directions, not merely the size of the acceleration.

\begin{theo}\label{th1}
Assume the admissible class above. Fix $\gamma\ge0$ and $\theta_*\in U$,
and suppose $G_\gamma(\theta_*)$ is positive definite. Let
$s_*=\lambda_{\min}(G_\gamma(\theta_*))^{1/2}$. There is a ball
$B_r(\theta_*)=\{\theta:|\theta-\theta_*|<r\}\subset U$ such that
\begin{equation}\label{eq:stability-estimate}
|\theta-\widetilde\theta|
\le\frac{2}{s_*}
\|\mathcal F_\gamma(\theta)-\mathcal F_\gamma(\widetilde\theta)\|_{\mathcal Y}
\end{equation}
for every $\theta,\widetilde\theta\in B_r(\theta_*)$. Consequently,
\begin{align}
\|\rho_\theta-\rho_{\widetilde\theta}\|_{L^\infty}
&\le\frac{2C_\rho}{s_*}\|\mathcal F_\gamma(\theta)-\mathcal F_\gamma(\widetilde\theta)\|_{\mathcal Y},\label{eq:density-stability}\\
\|f_\theta-f_{\widetilde\theta}\|_{H^4}
&\le\frac{2C_f}{s_*}\|\mathcal F_\gamma(\theta)-\mathcal F_\gamma(\widetilde\theta)\|_{\mathcal Y},\label{eq:initial-stability}
\end{align}
where
\[
C_\rho=\left(1+\sum_{j=1}^{p-1}\|r_j\|_{L^\infty}^2\right)^{1/2},
\qquad C_f=\left(\sum_{k=1}^q\|\phi_k\|_{H^4}^2\right)^{1/2}.
\]
\end{theo}

{The rank condition in Theorem~\ref{th1} must be checked for the chosen
profiles and reference state.} The next result gives a specific class
where identifiability is proved directly and a global estimate holds
on a compact parameter set.

\begin{theo}\label{th2}
Let $\Omega=(-1,1)$ and let $\phi$ be a real $L^2$-normalized even
clamped eigenfunction specified in \eqref{eq:explicit-mode}, with
$L\phi=\lambda\phi$. Set $g=0$, $\rho=a$, and $f=A\phi$. Fix
\[
0<a_-<a_+<\infty,\quad 0<A_-<A_+<\infty,\quad
K=[a_-,a_+]\times[A_-,A_+].
\]
For any $T>0$ and $\gamma_0\ge0$, there exists
$C=C(T,\gamma_0,K,\phi)>0$ such that, for every $\gamma\in[0,\gamma_0]$
and $(a,A),(\widetilde a,\widetilde A)\in K$,
\begin{equation}\label{eq:global-modal}
|a-\widetilde a|+\|(A-\widetilde A)\phi\|_{H^4(-1,1)}
\le C\|u_{xx}(1,\cdot;a,A,\gamma)
-u_{xx}(1,\cdot;\widetilde a,\widetilde A,\gamma)\|_{L^2(0,T)}.
\end{equation}
Thus a single curvature trace suffices in this particular class.
\end{theo}

{The nonzero amplitude bound excludes the zero solution, which contains
no density information. Theorem~\ref{th2} is a special case of the
regularity class above with $p=q=1$, but its proof establishes
injectivity rather than assuming it.} Neither theorem gives a universal
square-root growth law in $\gamma$. Uniformity on a bounded damping
interval and a sharp dependence as $\gamma\to\infty$ are different
questions.

\subsection{Related problems and organization}
Passive recovery of coefficients and sources for the plate equation is
studied in \cite{MR4598840}. Related recovery results for evolution and
wave equations appear in \cite{MR4959941,Kian2025WaveEquation}, and
inclusion recovery in photoacoustic imaging is considered in
\cite{Faouzi2025}. Simultaneous coefficient determination for the
nonlinear Westervelt equation is studied in \cite{Acosta2025Westervelt}
using boundary-excited time-periodic solutions. These models and data
sets differ from the passive, clamped evolution problem considered here.

Other biharmonic inverse problems include recovery of an inhomogeneity
from surface elastic measurements \cite{Ramm1988}, shape reconstruction
from partial Cauchy data \cite{Bouslah2024}, regularization of ill-posed
boundary or initial problems \cite{Benrabah2019,Danh2019}, and inverse
boundary problems for nonlinear perturbations or potentials
\cite{Bhattacharyya2025,Liu2025FourthOrderSchrodinger,Zhao2025}.
Frequency-domain scattering is treated through boundary integral
formulations \cite{dong2024novel} and cavity reconstruction
\cite{dong2024uniqueness,harris2025direct,li2026sampling}.
Random-source recovery \cite{Li2022a,li2026stability} and nonradiating
sources \cite{li2025nonradiating} provide further distinct settings.
None of these references is invoked to establish the joint sensitivity
condition in Theorem~\ref{th1}; that condition remains an explicit
hypothesis for the prescribed spatially varying profiles.

Section~\ref{sect2} proves the direct problem in the natural energy
space. Section~\ref{sec3} establishes strong parameter differentiability
and the sensitivity equations. Section~\ref{sec4} proves the inverse
estimates, including the explicit identifiable class of
Theorem~\ref{th2}. Section~\ref{sec5} presents the independent computations, including
noise repetitions, discretization checks, and the limitations of the
additional square-plate example. Section~\ref{sec6} summarizes the scope.
\section{The evolution operator in the natural energy space}\label{sect2}

In this section, $\rho\in L^\infty(\Omega)$ is any fixed density
satisfying $0<\rho_{\min}\le\rho\le\rho_{\max}$ almost everywhere,
and $\gamma\ge0$ is fixed. Set $V=H_0^2(\Omega)$, the closure of
$C_c^\infty(\Omega)$ in $H^2(\Omega)$, and
$H_\rho=L^2(\Omega;\rho\,dx)$. Define
\[
\mathcal H_\rho=V\times H_\rho,\qquad
\langle(u,v),(z,w)\rangle_{\mathcal H_\rho}
=\int_\Omega\Delta u\,\overline{\Delta z}\,dx
 +\int_\Omega\rho v\overline w\,dx.
\]
All inner products are linear in the first argument. The norm
$\|\Delta u\|_{L^2}$ is equivalent to $\|u\|_{H^2}$ on $V$.
Indeed, the zero extension of $u\in V$ belongs to $H^2(\mathbb R^n)$;
the Fourier identity for second derivatives and Poincar\'e inequalities
give the lower bound, and the upper bound is immediate. Thus
$\mathcal H_\rho$ is a Hilbert space, with norm equivalent to
$H_0^2(\Omega)\times L^2(\Omega)$.

Let $L$ be the clamped biharmonic operator, defined variationally by
\[
D(L)=\left\{u\in V:\text{ there exists }h\in L^2(\Omega)
\text{ with }\int_\Omega\Delta u\overline{\Delta z}
=\int_\Omega h\overline z\text{ for all }z\in V\right\},\quad Lu=h.
\]
On the smooth domain assumed here, clamped elliptic regularity gives
\begin{equation}\label{eq:elliptic}
D(L)=H^4(\Omega)\cap H_0^2(\Omega),\qquad
\|u\|_{H^4}\le C_\Omega\|Lu\|_{L^2};
\end{equation}
see \cite{Gazzola2010Polyharmonic}, Section~2.5 (the Dirichlet
regularity result). The absence of a lower-order term
follows from the trivial clamped kernel and the usual elliptic estimate.

Define
\[
\mathcal A_{\rho,\gamma}(u,v)
=(v,-\rho^{-1}Lu-\gamma\rho^{-1}v),\qquad
D(\mathcal A_{\rho,\gamma})=D(L)\times V.
\]

Multiplication by $\rho^{-1}$ is used only in $L^2$, so this definition
allows discontinuous positive densities.

\begin{lemm}\label{lemm1}
The operator $\mathcal A_{\rho,\gamma}$ is densely defined and dissipative
on $\mathcal H_\rho$.
\end{lemm}
\begin{proof}
The set $C_c^\infty(\Omega)\times C_c^\infty(\Omega)$ lies in the
domain and is dense in $V\times H_\rho$ in its energy norm. For
$(u,v)\in D(L)\times V$, the definition of $L$ gives
\[
\operatorname{Re}\langle\mathcal A_{\rho,\gamma}(u,v),(u,v)
\rangle_{\mathcal H_\rho}
=\operatorname{Re}\left(\int_\Omega\Delta v\overline{\Delta u}
-\int_\Omega Lu\overline v\right)-\gamma\|v\|_{L^2}^2
=-\gamma\|v\|_{L^2}^2.
\]
\end{proof}

\begin{lemm}\label{lemm2}
For each $\zeta>0$, the range of $\zeta I-\mathcal A_{\rho,\gamma}$
is $\mathcal H_\rho$. Consequently $\mathcal A_{\rho,\gamma}$ generates
a strongly continuous contraction semigroup on $\mathcal H_\rho$.
\end{lemm}
\begin{proof}
Given $(y_1,y_2)\in V\times H_\rho$, the resolvent equations reduce to
\[
v=\zeta u-y_1,\qquad
Lu+(\zeta^2\rho+\zeta\gamma)u
=\rho y_2+(\zeta\rho+\gamma)y_1.
\]
The associated form on $V$ is bounded and coercive. Lax--Milgram gives
a unique $u\in V$. The equation then gives $Lu\in L^2$, so $u\in D(L)$,
and $v\in V$. This proves surjectivity. Together with Lemma~\ref{lemm1},
the Lumer--Phillips theorem yields the semigroup assertion; see
\cite{Pazy1983}, Chapter~1.
\end{proof}

\begin{theo}\label{th:wellposed}
For each $(f,g)\in\mathcal H_\rho$, problem \eqref{a} has a unique
energy solution satisfying
\[(u,u_t)\in C([0,\infty);\mathcal H_\rho).\]
If $(f,g)\in D(L)\times V$, then
\[
u\in C([0,\infty);D(L))\cap C^1([0,\infty);V)
\cap C^2([0,\infty);L^2(\Omega)).
\]
For either class the natural energy satisfies

\begin{equation}\label{eq:energy}
E(t)=\frac12\int_\Omega(\rho|u_t|^2+|\Delta u|^2)\,dx,
\qquad E(t)+\gamma\int_0^t\|u_t(s)\|_{L^2}^2\,ds=E(0).
\end{equation}

\end{theo}
\begin{proof}
Existence and uniqueness follow from Lemma~\ref{lemm2}. The semigroup
leaves its generator domain invariant and commutes with the generator
there, yielding the stated strong-solution regularity. Differentiating
the squared energy norm for strong solutions gives \eqref{eq:energy} by
Lemma~\ref{lemm1}. Approximation in the energy norm extends the identity
to energy solutions.
\end{proof}

For fixed $\rho,\gamma$, write $\mathcal A=\mathcal A_{\rho,\gamma}$.
The semigroup is contractive in its generator graph norm, defined for
$W\in D(\mathcal A)$ by
\[
\|W\|_{D(\mathcal A)}^2
=\|W\|_{\mathcal H_\rho}^2+\|\mathcal A W\|_{\mathcal H_\rho}^2.
\]
This norm is equivalent to the $H^4\times H^2$ norm on $D(L)\times V$;
\eqref{eq:elliptic} and the formula for $\mathcal A$ prove both bounds.
The equivalence constants can be chosen uniformly for
$0\le\gamma\le\gamma_0$ and densities satisfying \eqref{eq:density-bounds}.
{This is the appropriate higher-regularity estimate. The quantity
$\|\Delta^2u\|_{L^2}^2+\int_\Omega\rho|u_t|^2\,dx$ is not the natural dissipative
energy.} No boundary condition on $\Delta^2f$ is required for membership
in $D(L)\times V$; for $f\in H^4$, its boundary trace need not even exist.

\section{Energy estimates and boundary sensitivities}\label{sec3}
\subsection{Integrated energy and forced solutions}

\begin{prop}\label{prop:integrated}
Let $0\le\gamma\le\gamma_0$. The energy in \eqref{eq:energy} satisfies
\begin{equation}\label{eq:integrated}
\int_0^T E(t)\,dt\ge\beta(T,\gamma_0)E(0),\qquad
\beta(T,\gamma_0)=
\begin{cases}
\dfrac{\rho_{\min}}{2\gamma_0}
(1-e^{-2\gamma_0T/\rho_{\min}}),&\gamma_0>0,\\[4pt]
T,&\gamma_0=0.
\end{cases}
\end{equation}
\end{prop}

\begin{proof}
Since $\|u_t\|_{L^2}^2\le2E(t)/\rho_{\min}$,
$E'(t)\ge-2\gamma_0E(t)/\rho_{\min}$. Multiplication by the
corresponding integrating factor and integration in time give the claim.
\end{proof}

{This lower bound does not establish boundary observability. For example,
combining it with an estimate
$\int_0^TE\le c_0E(0)+c_1\|\mathcal Bu\|_{\mathcal Y}^2$ would require
$\beta(T,\gamma_0)>c_0$ to absorb the initial-energy term. Merely bounding
$\gamma$ or subdividing time does not establish this inequality.}

For a source $F\in L^1(0,T;L^2(\Omega))$, the forced equation
$\rho u_{tt}+Lu+\gamma u_t=F$ has a mild energy solution. Its energy obeys

\begin{equation}\label{eq:forced-energy}
E(t)^{1/2}\le E(0)^{1/2}
 +\frac{1}{\sqrt{2\rho_{\min}}}\int_0^t\|F(s)\|_{L^2}\,ds.
\end{equation}

For smooth solutions this follows by testing with $u_t$; approximation,
or the variation-of-constants formula in $\mathcal H_\rho$, gives the
general case.

For two strong solutions with the same $\gamma$, put
$v=u_1-u_2$, $r=\rho_1-\rho_2$, $\delta f=f_1-f_2$, and
$\delta g=g_1-g_2$. For the inverse problem, $g_1=g_2=g$, so
$\delta g=0$.
Then
\begin{equation}\label{eq:difference}
\rho_1v_{tt}+Lv+\gamma v_t=-r u_{2,tt},\qquad
v(0)=\delta f,\quad v_t(0)=\delta g,
\end{equation}
with homogeneous clamped conditions. Define
\[
E_v(t)=\frac12\int_\Omega(\rho_1|v_t|^2+|\Delta v|^2)\,dx.
\]
If
$\|u_{2,tt}\|_{L^\infty(0,T;L^2)}\le M$, then
\begin{align}
E_v(t)^{1/2}&\le E_v(0)^{1/2}
 +\frac{Mt}{\sqrt{2\rho_{\min}}}\|r\|_{L^\infty},\label{eq:diff-energy}\\
\|r u_{2,tt}\|_{L^1(0,T;L^2)}&\le MT\|r\|_{L^\infty}.
\label{eq:source-bound}
\end{align}

{In particular, no small source coefficient follows from increasing $T$.}
At $t=0$, the identity
\[
\rho_1v_{tt}(0)+L\delta f+\gamma\delta g=-r u_{2,tt}(0)
\]
does not control $r$ by $\delta f$ and $\delta g$ alone, because different densities
may have the same initial state. Our inverse argument therefore uses
the full sensitivity map instead of such an initial-time decoupling.

\subsection{The observation topology}
The trace theorem on the smooth boundary gives
\begin{equation}\label{eq:trace}
\|\Delta w\|_{L^2(\partial\Omega)}
+\|\partial_\nu\Delta w\|_{L^2(\partial\Omega)}
\le C_\Omega\|w\|_{H^4(\Omega)}.
\end{equation}
Hence the observations of every strong solution of
Theorem~\ref{th:wellposed} belong to $\mathcal Y$. This assertion is
made in the generator domain, not for arbitrary natural-energy data.

More generally, let $\mathcal Z$ be a Hilbert space of measured data
and let $P:\mathcal Y\to\mathcal Z$ be a fixed bounded
linear measurement operator, such as spatial averaging followed by a
known bounded temporal filter. The local proof in Section~\ref{sec4}
also applies to $P\mathcal F_\gamma$ if its own derivative is injective.
Injectivity of $D\mathcal F_\gamma$ alone does not imply injectivity after
this change of measurement. {The actual sensor operator must therefore
be included when assessing numerical sensitivities.}

\subsection{Strong forced estimates and parameter differentiability}
For brevity write $\mathcal A=\mathcal A_{\rho,\gamma}$,
$\mathcal T(t)=e^{t\mathcal A}$,
$C H^s=C([0,T];H^s(\Omega))$, and
$C L^2=C([0,T];L^2(\Omega))$. On compact parameter sets and bounded damping
intervals, all constants in this subsection are uniform.

\begin{lemm}\label{lemm:strong-forcing}
Let $F\in C^1([0,T];L^2(\Omega))$ and
$W_0\in D(\mathcal A)$. The solution of
$\rho w_{tt}+Lw+\gamma w_t=F$ with homogeneous clamped conditions
and initial state $W(0)=W_0$, where $W=(w,w_t)$, satisfies
\begin{equation}\label{eq:strong-forcing}
\|W\|_{C([0,T];D(\mathcal A))}
\le C_T\left(\|W_0\|_{D(\mathcal A)}
+\|F\|_{C([0,T];L^2)}+\|F_t\|_{L^1(0,T;L^2)}\right).
\end{equation}
\end{lemm}
\begin{proof}
Put $H(t)=(0,\rho^{-1}F(t))$. Variation of constants gives
$W(t)=\mathcal T(t)W_0+\int_0^t \mathcal T(t-s)H(s)\,ds$. Integration by parts in the
semigroup convolution yields
\[
\mathcal A\int_0^t \mathcal T(t-s)H(s)\,ds
=\mathcal T(t)H(0)-H(t)+\int_0^t \mathcal T(t-s)H_t(s)\,ds.
\]
The identity follows first for domain-valued smooth $H$, and then for
$C^1$ Hilbert-space-valued $H$ by bounded resolvent approximation and
closedness of $\mathcal A$. Contractivity bounds both this expression
and the original convolution. Norm equivalence and the density bounds
give \eqref{eq:strong-forcing}. Here
$H_t=(0,\rho^{-1}F_t)$ because $\rho$ is independent of time, and
$\|H_t\|_{\mathcal H_\rho}\le \rho_{\min}^{-1/2}\|F_t\|_{L^2}$.
\end{proof}

The assumptions of Section~\ref{sect1} supply the time regularity
needed in this lemma. Indeed, for every integer $m\ge1$,
\begin{equation}\label{eq:compatibility}
(f_\theta,g)\in D(\mathcal A_{\rho_\theta,\gamma}^m).
\end{equation}
To verify this, all coefficients and initial profiles are smooth on
$\overline\Omega$. In the fixed collar where $\rho_\theta=a_0$, repeated
application of $\mathcal A$ to the initial pair leaves each component
in the finite span of the $\phi_k$. Both clamped traces therefore vanish
at every step. In the interior repeated application gives smooth
functions. This proves \eqref{eq:compatibility} by induction. Every
fixed-order graph norm is uniformly bounded on compact parameter sets,
because only finitely many smooth coefficient derivatives and reciprocal
powers of the positive density occur. The semigroup estimate applied to
$\mathcal A^m(f_\theta,g)$ then bounds all time derivatives required
below, in particular $u_{tt},u_{ttt},u_{tttt}$ in $C([0,T];L^2)$.

For $h=(\alpha_0,\ldots,\alpha_{p-1},\eta_1,\ldots,\eta_q)$ set
\[
r_h=\alpha_0+\sum_{j=1}^{p-1}\alpha_jr_j,\qquad
f_h=\sum_{k=1}^q\eta_k\phi_k.
\]
For a fixed base point $(\theta,\gamma)$, define
$z_h=z_h[\theta,\gamma]$ as the strong solution of
\begin{equation}\label{eq:sensitivity}
\begin{cases}
\rho_\theta z_{h,tt}+Lz_h+\gamma z_{h,t}=-r_h u_{\theta,\gamma,tt},\\
z_h(0)=f_h,\quad z_{h,t}(0)=0,\\
z_h=\partial_\nu z_h=0\quad\text{on }\Sigma_T.
\end{cases}
\end{equation}
We suppress the base point when it is fixed.
Its existence with $z_h\in C([0,T];H^4)$ follows from
Lemma~\ref{lemm:strong-forcing}. It depends linearly on $h$.

\begin{prop}\label{prop:differentiability}
For the admissible class in Section~\ref{sect1}, the state map is $C^1$ into
\[C([0,T];H^4(\Omega)\cap H_0^2(\Omega)),\]
with derivative $h\mapsto z_h$.
Its derivative is locally Lipschitz in $\theta$. The derivative of
$\mathcal F_\gamma$ is $h\mapsto\mathcal Bz_h$ and is jointly
continuous in $(\gamma,\theta)$ on bounded damping intervals.
\end{prop}
\begin{proof}
Work in a compact parameter neighborhood contained in $U$ and fix
$0\le\gamma\le\gamma_0$. Write $u=u_{\theta,\gamma}$,
$u'=u_{\theta+h,\gamma}$, $w=u'-u$, and $\rho'=\rho_\theta+r_h$.
Subtraction gives
\[
\rho'w_{tt}+Lw+\gamma w_t=-r_hu_{tt},\qquad
w(0)=f_h,\quad w_t(0)=0.
\]
The source and its time derivative are $O(|h|)$ in
$C([0,T];L^2)$. Lemma~\ref{lemm:strong-forcing} gives
\[\|w\|_{C H^4}+\|w_t\|_{C H^2}\le C|h|.\]

To obtain the additional derivative estimate, differentiate this
equation once in time. The initial pair for $w_t$ is
\[
(w_t(0),w_{tt}(0))
=\left(0,-(\rho')^{-1}(Lf_h+r_hu_{tt}(0))\right).
\]
Its second component belongs to $H_0^2$ and is $O(|h|)$ there.
Indeed, all factors are smooth, and in the boundary collar it is a
linear combination of the clamped eigenfunctions. Uniform $W^{2,\infty}$
bounds on the coefficients give the quantitative bound. The differentiated
source and its derivative are $-r_hu_{ttt}$ and $-r_hu_{tttt}$, respectively.
Another application of \eqref{eq:strong-forcing}, followed by the
differentiated equation, gives
\begin{equation}\label{eq:diff-strong}
\|w\|_{C H^4}+\|w_{tt}\|_{C L^2}
+\|w_{ttt}\|_{C L^2}\le C|h|.
\end{equation}

Set $R=w-z_h$ and rewrite the difference equation using density
$\rho_\theta$. Then
\[
\rho_\theta R_{tt}+LR+\gamma R_t=-r_hw_{tt},\qquad
R(0)=R_t(0)=0.
\]
The source and its first time derivative are $O(|h|^2)$ by
\eqref{eq:diff-strong}. Hence
\begin{equation}\label{eq:quadratic-remainder}
\|u_{\theta+h,\gamma}-u_{\theta,\gamma}-z_h\|_{C H^4}
\le C|h|^2.
\end{equation}
This proves Fr\'echet differentiability.

For continuity of the derivative, fix $k\in\mathbb R^{p+q}$ and compare
the sensitivities at $\theta+h$ and $\theta$. Their difference
$v=z_k[\theta+h,\gamma]-z_k[\theta,\gamma]$
has zero initial data and satisfies
\[
\rho'v_{tt}+Lv+\gamma v_t=-r_k w_{tt}-r_hz_{k,tt}.
\]
For clarity, the time-differentiated sensitivity has initial pair
\[
(z_{k,t}(0),z_{k,tt}(0))
=\bigl(0,-\rho_\theta^{-1}(Lf_k+r_ku_{tt}(0))\bigr)
\in D(L)\times H_0^2.
\]
Its norm is $O(|k|)$ by the same collar argument. Its source and
source derivative are $-r_ku_{ttt}$ and $-r_ku_{tttt}$. Solving this
differentiated problem and integrating in time, uniqueness identifies
its solution with $z_{k,t}$.
Lemma~\ref{lemm:strong-forcing} and the differentiated equation thus give
$\|z_{k,tt}\|_{C L^2}+\|z_{k,ttt}\|_{C L^2}\le C|k|$.
Thus the source and its derivative are $O(|h||k|)$, giving
$\|v\|_{C H^4}\le C|h||k|$. For a simultaneous change $\kappa=\gamma'-\gamma$, the state difference
has the extra source $-\kappa u_t$. Its initial acceleration also has
$- (\rho')^{-1}\kappa g$, which lies in $H_0^2$ and is $O(|\kappa|)$
there. The same two forced estimates bound the state difference and
its second and third time derivatives by $C(|h|+|\kappa|)$.
The sensitivity difference equation then has the extra source
$-\kappa z_{k,t}$. Its source and first time derivative are bounded by
$C(|h|+|\kappa|)|k|$, proving joint continuity, including at $\gamma=0$. The trace bound
\eqref{eq:trace} proves the statements about $\mathcal F_\gamma$.
\end{proof}
\section{Lipschitz stability}\label{sec4}
\subsection{Joint sensitivity and local inversion}
Let $S=D\mathcal F_\gamma(\theta_*):\mathbb R^{p+q}\to\mathcal Y$.
We use $\mathcal L(X,Y)$ for the space of bounded linear operators
from $X$ to $Y$, with its operator norm. For real parameter directions,
\[
h^{\mathsf T}G_\gamma(\theta_*)h=\|Sh\|_{\mathcal Y}^2.
\]
Thus positive definiteness of \eqref{eq:gram-intro} is equivalent to
injectivity of $S$, and
\[
\min_{|h|=1}\|Sh\|_{\mathcal Y}
=\lambda_{\min}(G_\gamma(\theta_*))^{1/2}=s_*.
\]
Individual nonzero sensitivity columns do not imply this condition:
a density variation and an initial-state variation may cancel in the
linearized observation. The full Gram matrix is therefore required.

\begin{proof}[Proof of Theorem~\ref{th1}]
Proposition~\ref{prop:differentiability} gives a ball
$B_r(\theta_*)\subset U$ on which
\[
\sup_{\vartheta\in B_r(\theta_*)}
\|D\mathcal F_\gamma(\vartheta)-S\|_{\mathcal L(\mathbb R^{p+q},\mathcal Y)}
\le s_*/2.
\]
Let $h=\theta-\widetilde\theta$. Integration along the segment between
the two points gives
\[
\mathcal F_\gamma(\theta)-\mathcal F_\gamma(\widetilde\theta)
=Sh+\int_0^1
[D\mathcal F_\gamma(\widetilde\theta+th)-S]h\,dt.
\]
The reverse triangle inequality yields a lower bound
$(s_*/2)|h|$, proving \eqref{eq:stability-estimate}. The affine profile
representation and Cauchy--Schwarz give
\[
\|\rho_\theta-\rho_{\widetilde\theta}\|_{L^\infty}\le C_\rho|h|,
\qquad\|f_\theta-f_{\widetilde\theta}\|_{H^4}\le C_f|h|.
\]
These prove \eqref{eq:density-stability} and
\eqref{eq:initial-stability}.
\end{proof}
The small quantity in this proof is the nonlinear remainder in a
parameter neighborhood. No source term is absorbed by increasing $T$.
The $H^4$ estimate uses the prescribed initial profiles and does not
compare fourth-order derivatives with the natural energy.

\begin{coro}\label{cor:noise}
Let $y^\delta\in\mathcal Y$ be measured data with noise bound
$\delta\ge0$, and let $\varepsilon\ge0$ be a data-fit tolerance.
Suppose the true and reconstructed parameters
$\theta^\dagger,\widehat\theta\in B_r(\theta_*)$ satisfy
\[
\|y^\delta-\mathcal F_\gamma(\theta^\dagger)\|_{\mathcal Y}\le\delta,
\qquad\|\mathcal F_\gamma(\widehat\theta)-y^\delta\|_{\mathcal Y}\le\varepsilon.
\]
Then $|\widehat\theta-\theta^\dagger|\le2s_*^{-1}(\delta+\varepsilon)$.
The density and initial-state errors obey the corresponding estimates
with factors $C_\rho$ and $C_f$.
\end{coro}
\begin{proof}
Apply the triangle inequality to the observation difference and then
Theorem~\ref{th1}.
\end{proof}
This result does not assert that an optimizer enters the required local
neighborhood. A convergence claim for an algorithm is a separate issue.

\begin{prop}\label{prop:uniform}
Fix $\theta_*\in U$ and $\gamma_0\ge0$. If
$G_\gamma(\theta_*)$ is positive definite for every
$\gamma\in[0,\gamma_0]$, then
\[
\underline s=\min_{0\le\gamma\le\gamma_0}
\lambda_{\min}(G_\gamma(\theta_*))^{1/2}>0.
\]
Theorem~\ref{th1} and Corollary~\ref{cor:noise} hold on a common
parameter ball with $s_*$ replaced by $\underline s$.
\end{prop}

\begin{proof}
Joint continuity from Proposition~\ref{prop:differentiability} and
compactness of the damping interval give a strictly positive minimum.
Uniform continuity of the derivative on a compact parameter neighborhood
then supplies a common ball in the proof of Theorem~\ref{th1}.
\end{proof}
The condition must hold over the entire interval, not just at a finite
set of sampled damping values. Uniformity here does not identify a
power law as the upper damping bound tends to infinity.

\subsection{A class with explicit identifiability}
Fix an integer $m\ge0$ and choose the root
\[\tan\beta_m=-\tanh\beta_m,\qquad (m+1/2)\pi<\beta_m<(m+1)\pi.\]
Such a root exists by continuity and is unique because
$\tan x+\tanh x$ is strictly increasing on this interval. Define
\begin{equation}\label{eq:explicit-mode}
\psi(x)=\cos(\beta_m x)-\frac{\cos\beta_m}{\cosh\beta_m}\cosh(\beta_m x),
\qquad \phi=\psi/\|\psi\|_{L^2(-1,1)},\qquad\lambda=\beta_m^4.
\end{equation}
The usual even clamped beam modes have this form; see also
\cite{WangChristov2024} (Section~2 of its linked preprint). Direct differentiation, sufficient
for the argument here, verifies
\[
\phi''''=\lambda\phi,\qquad\phi(\pm1)=\phi'(\pm1)=0,
\qquad c_\phi:=\phi''(1)
=-\frac{2\beta_m^2\cos\beta_m}{\|\psi\|_{L^2(-1,1)}}\ne0.
\]
Indeed, $\sin\beta_m+\cos\beta_m\tanh\beta_m=0$ gives the derivative
boundary condition, and $\cos\beta_m\ne0$ for such a root.

For constant density $a>0$, $g=0$, and $f=A\phi$, separation of
variables gives
\begin{equation}\label{eq:modal-ode}
u(x,t)=Aq_{a,\gamma}(t)\phi(x),\qquad
a q''+\gamma q'+\lambda q=0,\quad q(0)=1,\quad q'(0)=0.
\end{equation}
Here $q=q_{a,\gamma}$ is the scalar temporal profile, and primes
in its equation denote time derivatives.
The measured curvature is
$\mathcal O_\gamma(a,A)=c_\phi A q_{a,\gamma}\in L^2(0,T)$.
The matrix-exponential representation of the two-dimensional ordinary
differential equation (ODE)
shows that $\mathcal O_\gamma$ is smooth in $(a,A,\gamma)$ for $a>0$,
also at critical damping.

\begin{lemm}\label{lemm:modal-injectivity}
For each $\gamma\ge0$ and $T>0$, the map $\mathcal O_\gamma$ is
injective on $\{a>0,A>0\}$, and its derivative is injective at every
point of that set.
\end{lemm}
\begin{proof}
Equality of two observations in $L^2(0,T)$ implies equality everywhere
on $[0,T]$ because the observations are continuous. Their values at
$t=0$ give $c_\phi A=c_\phi\widetilde A$, hence
$A=\widetilde A$. The functions are smooth, so equality also implies
equality of their second derivatives at zero. Since
$q_{a,\gamma}''(0)=-\lambda/a$, it follows that $a=\widetilde a$.

For a direction $(h_a,h_A)$, divide the derivative observation by
$c_\phi$ and write
\[
Z(t)=h_Aq_{a,\gamma}(t)+Ah_a\partial_aq_{a,\gamma}(t).
\]
If $Z=0$ in $L^2(0,T)$, smoothness gives
\[
0=Z(0)=h_A,\qquad
0=Z''(0)=-\frac{\lambda}{a}h_A+\frac{A\lambda}{a^2}h_a.
\]
Thus $h_A=h_a=0$. Endpoint evaluations are used only to prove
injectivity of smooth functions; no bound of a pointwise time derivative
by an arbitrary $L^2$ norm is used.
\end{proof}

\begin{proof}[Proof of Theorem~\ref{th2}]
Let $K=[a_-,a_+]\times[A_-,A_+]$. By smoothness, compactness, and
Lemma~\ref{lemm:modal-injectivity},
\[
s_0=\min_{\substack{\gamma\in[0,\gamma_0],\ \theta\in K\\|h|=1}}
\|D\mathcal O_\gamma(\theta)h\|_{L^2(0,T)}>0.
\]
Uniform continuity of the derivative gives $r_0>0$ such that for
$\theta,\widetilde\theta\in K$ with $|\theta-\widetilde\theta|<r_0$,
integration along their segment and the reverse triangle inequality give
\[
\|\mathcal O_\gamma(\theta)-\mathcal O_\gamma(\widetilde\theta)\|_{L^2}
\ge (s_0/2)|\theta-\widetilde\theta|
\]
uniformly in $\gamma$. The segment lies in $K$, which is convex.

For the remaining pairs, if any, continuity and injectivity imply
\[
m_0=\min_{\substack{\gamma\in[0,\gamma_0],\ \theta,\widetilde\theta\in K\\
|\theta-\widetilde\theta|\ge r_0}}
\|\mathcal O_\gamma(\theta)-\mathcal O_\gamma(\widetilde\theta)\|_{L^2}>0.
\]
On that compact set,
$|\theta-\widetilde\theta|\le\operatorname{diam}(K)m_0^{-1}
\|\mathcal O_\gamma(\theta)-\mathcal O_\gamma(\widetilde\theta)\|_{L^2}$.
Combine the two estimates. Finally,
\[
|a-\widetilde a|+|A-\widetilde A|\|\phi\|_{H^4}
\le(1+\|\phi\|_{H^4}^2)^{1/2}|\theta-\widetilde\theta|.
\]
This proves \eqref{eq:global-modal} with the stated uniform dependencies.
\end{proof}

The argument proves existence of a uniform constant for each fixed
$T>0$ and bounded damping interval. It does not supply a sharp constant
or a uniform bound as $T\downarrow0$ or $\gamma_0\to\infty$.

\subsection{Verification of the rank condition in other parameter classes}
For spatially varying densities in Theorem~\ref{th1}, one may compute
the sensitivity solutions and their Gram matrix. To turn a numerical
calculation into a continuous rank certificate, an error estimate is
needed. Let $S=D\mathcal F_\gamma(\theta_*)$ and let $S_{\rm disc}$
be an approximation, with both operators mapping
$\mathbb R^{p+q}$ into the same data Hilbert space.
If $\|S_{\rm disc}-S\|\le\varepsilon_{\rm disc}$ in operator norm, then
\[
\min_{|v|=1}\|Sv\|_{\mathcal Y}
\ge\min_{|v|=1}\|S_{\rm disc}v\|_{\mathcal Y}-\varepsilon_{\rm disc}.
\]
A positive right-hand side certifies injectivity. A positive singular
value on a mesh without a corresponding error bound is numerical
evidence only. No such certificate is claimed for the retained
experiments in Section~\ref{sec5}.

\section{Reproducible numerical experiments}\label{sec5}
The accompanying Python script
specifies the forward solvers, observations, noise, and optimization,
and saves every reconstructed parameter pair and noise realization.
The first experiment belongs to the class of Theorem~\ref{th2}. The
second is an exploratory square-plate experiment with a discontinuous
density and explicitly defined averaged, finite-bandwidth observations.
It is not claimed to satisfy the smooth-domain hypotheses of
Theorem~\ref{th1}.

\subsection{Noise, normalization, and reconstruction}
In both experiments the true parameter pair is $(a^*,A^*)=(2,100)$ and
the initial velocity is zero. We use
$\gamma\in\{0,1,2,5,10,15,20\}$ and perform 20 independent noise
trials for each value. Let $y_\ell^*$ denote a noiseless data block. In one dimension there
is a single block. In two dimensions $\ell\in\{m,s\}$ denotes moment
or shear, and each block stacks the observations of three sensors
on each of four edges. Let $\mathbf W$ be the diagonal matrix of trapezoidal time weights,
repeated for all sensors and edges in the block. Independently generated
standard Gaussian vectors $\xi_\ell$ are rescaled as
\begin{equation}\label{eq:numerical-noise}
\eta_\ell=0.05\,\frac{\|y_\ell^*\|_{\mathbf W}}{\|\xi_\ell\|_{\mathbf W}}\xi_\ell,
\qquad d_\ell=y_\ell^*+\eta_\ell,\qquad
\|v\|_{\mathbf W}^2=v^{\mathsf T}\mathbf Wv.
\end{equation}
{The 5\% noise level applies to each complete data block, not separately
to each sensor.} The vectors are Gaussian before normalization;
the normalized entries are not independent Gaussian variables.
In the plate experiment noise is added after the sensor response.

The least-squares objective is
\begin{equation}\label{eq:numerical-cost}
J(a,A)=\sum_\ell\frac{\|y_\ell(a,A)-d_\ell\|_{\mathbf W}^2}{\|y_\ell^*\|_{\mathbf W}^2}.
\end{equation}
Here $y_\ell(a,A)$ is the predicted data block at the trial parameters.
The block scales are fixed during optimization. They are known
calibration quantities in this synthetic experiment; no candidate
parameter is used to update them. This weights the moment and shear
data according to their prescribed relative noise levels.

For a fixed $a$, the observation is linear in $A$. Let $\mathbf v(a)$
and $\mathbf d$ be the vectors obtained by stacking, respectively,
$\mathbf W^{1/2}y_\ell(a,1)/\|y_\ell^*\|_{\mathbf W}$ and
$\mathbf W^{1/2}d_\ell/\|y_\ell^*\|_{\mathbf W}$ over the data blocks.
The amplitude minimizer is
\[
A(a)=\operatorname{proj}_{[50,150]}
\frac{\mathbf v(a)^{\mathsf T}\mathbf d}{\mathbf v(a)^{\mathsf T}\mathbf v(a)}.
\]
Here $\operatorname{proj}_{[50,150]}$ is the Euclidean projection onto
the indicated interval; the formula applies when $\mathbf v(a)\ne0$.
We first scan a uniform density grid and then minimize the reduced
scalar objective within the adjacent grid interval around its best
point, using a bounded scalar search with absolute tolerance $10^{-8}$.
This is the implemented search procedure, not a claim of certified
global optimization for the plate problem.

For each trial we report
\begin{equation}\label{eq:2derror}
e_a=\frac{|a_{\rm rec}-a^*|}{a^*},\qquad
e_A=\frac{|A_{\rm rec}-A^*|}{A^*},\qquad
\mathcal E_{\rm rel}=\sqrt{e_a^2+e_A^2}.
\end{equation}
The subscript $\mathrm{rec}$ denotes a reconstructed value, and a star
denotes the corresponding true value. The tables use percentages;
$\overline e_a$ and $\overline e_A$ are trial means, and the combined
error is shown as the mean and sample standard deviation of
$\mathcal E_{\rm rel}$. Both components are relative errors;
{there is no additional factor $\sigma$.} Write $f=A f_{\rm ref}$,
where the fixed profile $f_{\rm ref}$ is $\phi$ in the one-dimensional
experiment and $b(x)b(y)$ in \eqref{eq:2dparameters}. Then
\begin{equation}\label{eq:relative-profile}
\frac{\|f_{\rm rec}-f^*\|_{H^s}}{\|f^*\|_{H^s}}=e_A,
\qquad s=2,4.
\end{equation}
This is an amplitude-only error, not an assessment of arbitrary-profile
reconstruction. Parameter errors below the data-noise level are possible
because two unknown parameters are estimated from many observations.

The base random seed is 20260907. With damping index $j=0,\ldots,6$,
the seeds are $20260907+1000+j$ in the first experiment and
$20260907+2000+j$ in the second. Full arrays and software versions are
included with the code.

\subsection{One-dimensional experiment in the proved class}
We take the first even mode in \eqref{eq:explicit-mode}, corresponding
to the root in $(\pi/2,\pi)$, on $\Omega=(-1,1)$. With $\rho=a$ and
$f=A\phi$, the forward solution is exactly \eqref{eq:modal-ode}.
The observation is the single curvature trace at $x=1$. Dividing by
the known nonzero factor $c_\phi$ gives $Aq_{a,\gamma}(t)$ without
changing a relative-noise experiment.

We use $T=2$ and 1001 equally spaced time samples. The oscillator is
evaluated analytically, with separate underdamped, critically damped,
and overdamped formulas. An independent matrix-exponential check,
including critical damping, has maximum absolute discrepancy below
$4\times10^{-15}$. There is no spatial truncation in this experiment.
The density scan contains 81 points in $[1,3]$.
In the reconstruction tables, $\widehat s_{\min}$ denotes the smallest
singular value of the weighted discrete sensitivity matrix at the truth,
using relative parameter coordinates $(a/a^*,A/A^*)$.

\begin{table}[htbp]
\centering\small
\caption{One-dimensional exact-mode reconstruction: 20 trials with 5\% relative noise. The combined error is shown as mean $\pm$ sample standard deviation.}
\label{tab:gamma_error}
\setlength{\tabcolsep}{4pt}
\begin{tabular}{rrrrrrr}
\toprule
$\gamma$ & Mean $a_{\rm rec}$ & Mean $A_{\rm rec}$ & $\overline e_a$ (\%) & $\overline e_A$ (\%) & $\mathcal E_{\rm rel}$ (\%) & $\widehat s_{\min}$\\
\midrule
0 & 1.99945 & 99.9556 & 0.0504 & 0.1122 & $0.1323\pm0.1086$ & 0.9931\\
1 & 1.99988 & 99.9652 & 0.0721 & 0.1348 & $0.1617\pm0.0713$ & 0.9800\\
2 & 1.99985 & 100.0301 & 0.0944 & 0.1408 & $0.1836\pm0.1188$ & 0.9567\\
5 & 1.99947 & 99.9870 & 0.2016 & 0.1165 & $0.2531\pm0.1647$ & 0.7171\\
10 & 1.99935 & 99.9713 & 0.5285 & 0.1354 & $0.5687\pm0.4349$ & 0.2815\\
15 & 2.00943 & 99.9337 & 0.8889 & 0.1679 & $0.9280\pm0.5965$ & 0.1343\\
20 & 2.00834 & 99.9463 & 1.2344 & 0.1882 & $1.2791\pm0.8497$ & 0.0743\\
\bottomrule
\end{tabular}
\end{table}

The amplitude column of this matrix is exact, and the density column
uses a centered difference of step $2\times10^{-5}$.
Halving this step changes the sensitivity matrix by less than
$1.5\times10^{-10}$ in relative norm. These numerical singular
values illustrate conditioning; the continuous injectivity statement
is supplied by Theorem~\ref{th2}, not inferred from the table.
Noiseless recovery has maximum absolute density and amplitude errors
below $1.4\times10^{-8}$ and $2.3\times10^{-7}$,
respectively.

The mean density error rises as damping increases in this experiment,
while the amplitude error remains smaller. The declining sensitivity
singular value indicates poorer local conditioning under the fixed
normalization used here. Neither these data nor the theorems establish
a prescribed power law in the damping coefficient.

\subsection{A two-dimensional finite-bandwidth experiment}
Consider the square-plate configuration
\begin{equation}\label{eq:2dparameters}
\begin{split}
\Omega&=(0,1)^2,\qquad
\omega=\{(x,y):(x-0.5)^2+(y-0.5)^2<0.2^2\},\\
\rho(x,y)&=1+(a-1)\mathbf1_\omega(x,y),\qquad
f(x,y)=A\,b(x)b(y),\quad b(x)=x^2(1-x)^2.
\end{split}
\end{equation}
Here $\mathbf1_\omega$ is the characteristic function of the disk
$\omega$.
{The square and the discontinuous density lie outside the smooth class
in Section~\ref{sect1}.} This example tests the specified numerical
measurement and reconstruction procedure, not the general applicability
of that theorem.

\subsubsection{Both boundary channels and the sensor response}
On each edge use three normalized nonnegative sensor weights
\[
w_j(s)=\frac{b(s)(2s-1)^{2j}}{\int_0^1b(r)(2r-1)^{2j}\,dr},
\qquad j=0,1,2.
\]
The measured channels are spatial averages of both $\Delta u$ and
$\partial_\nu\Delta u$. We call these the moment and generalized shear
channels, respectively. At the edge $x=1$, write them as
\[
m_j(t)=\int_0^1\Delta u(1,y,t)w_j(y)\,dy,
\qquad s_j(t)=\int_0^1\partial_\nu\Delta u(1,y,t)w_j(y)\,dy.
\]
Define $V_{0j}=x^2(x-1)w_j(y)$ and
$V_{1j}=(3x^2-2x^3)w_j(y)$. Green's identity yields
\begin{align}
m_j&=\int_\Omega\Delta u\Delta V_{0j}
 +\int_\Omega(\rho u_{tt}+\gamma u_t)V_{0j},\label{eq:weak-moment}\\
s_j&=-\int_\Omega(\rho u_{tt}+\gamma u_t)V_{1j}
 -\int_\Omega\Delta u\Delta V_{1j}.\label{eq:weak-shear}
\end{align}
The test functions impose the required value or normal derivative only
on the selected edge and vanish with the complementary traces on the
other edges. For $u\in H_0^2(\Omega)$ with $Lu\in L^2(\Omega)$, the right-hand
sides also define generalized averaged tractions: changing an extension
while preserving its boundary value and normal derivative changes it
by an element of $H_0^2$, on which the variational identity makes the
difference zero. These identities define the discrete averaged tractions
from volume quantities. We do not take a raw third derivative of a
low-order displacement approximation. For the Galerkin approximation
below, these particular extensions are fixed; extension independence
is asserted only for the continuous variational solution.
A polynomial manufactured-field
check of both identities has maximum absolute error below
$7\times10^{-13}$.

Each averaged signal is passed through a specified causal sensor kernel
\[
k_\tau(t)=\frac{t}{\tau^2}e^{-t/\tau}\mathbf1_{t\ge0},\qquad\tau=0.02.
\]
This is two first-order low-pass stages with zero initial sensor states.
The two data blocks in \eqref{eq:numerical-cost} contain these filtered
sensor outputs. The bandwidth and spatial averaging are part of the
measurement operator.

\subsubsection{Conforming approximation and convergence checks}
Fix a truncation order $n_{\rm b}\ge1$.
The $n_{\rm b}$ spatial basis functions on $(0,1)$ are
\[
\chi_i(x)=\frac{b(x)P_{2i}^{(4,4)}(2x-1)}{\mathcal N_i},\qquad
\mathcal N_i^2=\frac{\Gamma(2i+5)^2}
{(4i+9)\Gamma(2i+1)\Gamma(2i+9)},\quad 0\le i<n_{\rm b}.
\]
Here $P_k^{(4,4)}$ denotes the Jacobi polynomial of degree $k$, and
$\Gamma$ is Euler's gamma function; $\mathcal N_i$ is chosen positive.
The basis is $L^2$-orthonormal and satisfies the clamped conditions.
For $0\le i\le j<n_{\rm b}$, define
\[
\Psi_{ij}(x,y)=
\begin{cases}
\chi_i(x)\chi_i(y),&i=j,\\
\bigl(\chi_i(x)\chi_j(y)+\chi_j(x)\chi_i(y)\bigr)/\sqrt2,&i<j.
\end{cases}
\]
Enumerate these $N=n_{\rm b}(n_{\rm b}+1)/2$ functions as
$\Psi_1,\ldots,\Psi_N$, with $\Psi_1=\Psi_{00}$.
Reflection and exchange symmetry are exact for \eqref{eq:2dparameters}.
Write
\[
u_N(x,y,t)=\sum_{i=1}^N q_i(t)\Psi_i(x,y),\qquad
\mathbf q=(q_1,\ldots,q_N)^{\mathsf T}.
\]
The semidiscrete system is
\[
\mathbf M(a)\ddot{\mathbf q}+\gamma\dot{\mathbf q}+\mathbf K\mathbf q=0,
\qquad \mathbf M(a)=\mathbf I_N+(a-1)\mathbf R,
\]
where $\mathbf I_N$ is the identity matrix,
$\mathbf R_{ij}=\int_\omega\Psi_i\Psi_j\,dx\,dy$, and
$\mathbf K_{ij}=\int_\Omega\Delta\Psi_i\Delta\Psi_j\,dx\,dy$.
The initial coefficient vector has first entry $A/630$ and all other
entries zero; initial velocity is zero.

Volume products use 64-point Gauss quadrature in each coordinate.
For disk integrals the polar quadrature counts are
\[
n_r=\max\{36,4n_{\rm b}+8\},\qquad
n_\vartheta=\max\{128,16n_{\rm b}\},
\]
with radial Gauss nodes and equally spaced angular nodes. The resulting basis
orthonormality error is below $8\times10^{-14}$ in the orders tested.
The finite-dimensional evolution and the sensor convolution are
evaluated from their exponential modes. A direct exponential of an
augmented oscillator--sensor system agrees with the evaluator to
relative error below $8\times10^{-15}$ in the independent check.
No explicit time-stepping or unspecified Courant--Friedrichs--Lewy
(CFL) condition is used.

Data are generated with $n_{\rm b}=16$ ($N=136$);
inversion uses $n_{\rm b}=8$ ($N=36$).
We use $T=0.3$ and 901 time samples. Table~\ref{tab:convergence} compares
the noiseless channel vectors with $n_{\rm b}=20$ at the true parameters. It
reports the largest relative discrepancy over $\gamma=0,5,20$.
These are finite refinement comparisons, not a proved continuum error
bound. Doubling the time sampling to 1801 changes the computed channel
norms by less than $4.4\times10^{-7}$ relative to the finer quadrature.

\begin{table}[htbp]
\centering\small
\caption{Maximum noiseless observation differences from $n_{\rm b}=20$ over $\gamma=0,5,20$, in percent.}
\label{tab:convergence}
\begin{tabular}{rrrr}
\toprule
$n_{\rm b}$ & $N=n_{\rm b}(n_{\rm b}+1)/2$ & Moment channel (\%) & Shear channel (\%)\\
\midrule
4 & 10 & 0.068736 & 0.233231\\
6 & 21 & 0.022841 & 0.095728\\
8 & 36 & 0.003010 & 0.015519\\
10 & 55 & 0.000625 & 0.004343\\
12 & 78 & 0.000307 & 0.002418\\
16 & 136 & 0.000058 & 0.000592\\
\bottomrule\end{tabular}\end{table}

The four edge signals coincide by symmetry. Gaussian readout noise is
generated independently on each edge before the block normalization
in \eqref{eq:numerical-noise}. For the block restricted to edge
$e=1,\ldots,4$, write $d_\ell^{(e)}$ and put
$\bar d_\ell=\frac14\sum_{e=1}^4d_\ell^{(e)}$.
Let $y_\ell^E(a,A)$ denote the common predicted single-edge vector
and $y_\ell^{E,*}$ its noiseless true counterpart. Put
$c_\ell=\|y_\ell^{E,*}\|_{\mathbf W_E}$, where $\mathbf W_E$ is the
time-weight matrix for the three sensors on one edge. Then
\[
J(a,A)=\sum_{\ell\in\{m,s\}}
\frac{\|y_\ell^E(a,A)-\bar d_\ell\|_{\mathbf W_E}^2}{c_\ell^2}+C_{\rm data},
\]
where $C_{\rm data}$ is independent of $(a,A)$. Thus fitting the edge averages
has the same minimizers as fitting all four edges.
{After edge averaging, the effective relative noise is approximately
2.5\%; the prescribed 5\% refers to the full blocks before averaging.}
Across the 140 plate trials, the mean effective relative noise is
$2.4977\%$ for moment and $2.4984\%$ for shear. These values are computed
from the saved noise realizations. The reduction follows from repeated
measurements of the same signal; it is not an additional rescaling.

The density search uses
41 points in $[1.2,2.8]$, followed by the scalar refinement described
above.

\begin{table}[htbp]
\centering\small
\caption{Two-dimensional averaged, filtered observations: 20 trials with 5\% relative noise in each full four-edge data block before edge averaging. Data and inversion use different spatial orders.}
\label{tab:2d_results}
\setlength{\tabcolsep}{4pt}
\begin{tabular}{rrrrrrr}
\toprule
$\gamma$ & Mean $a_{\rm rec}$ & Mean $A_{\rm rec}$ & $\overline e_a$ (\%) & $\overline e_A$ (\%) & $\mathcal E_{\rm rel}$ (\%) & $\widehat s_{\min}$\\
\midrule
0 & 2.00008 & 100.0054 & 0.0163 & 0.0288 & $0.0356\pm0.0228$ & 1.3985\\
1 & 1.99993 & 100.0072 & 0.0157 & 0.0275 & $0.0339\pm0.0146$ & 1.3944\\
2 & 2.00008 & 99.9967 & 0.0179 & 0.0306 & $0.0389\pm0.0197$ & 1.3897\\
5 & 1.99972 & 99.9974 & 0.0211 & 0.0264 & $0.0375\pm0.0178$ & 1.3710\\
10 & 2.00013 & 99.9985 & 0.0223 & 0.0246 & $0.0365\pm0.0213$ & 1.3173\\
15 & 2.00019 & 100.0037 & 0.0307 & 0.0272 & $0.0443\pm0.0237$ & 1.2175\\
20 & 2.00006 & 99.9989 & 0.0291 & 0.0218 & $0.0391\pm0.0283$ & 1.0686\\
\bottomrule
\end{tabular}
\end{table}

The noiseless cross-resolution inversions have maximum absolute density
and amplitude errors below $4.8\times10^{-6}$ and
$9.3\times10^{-5}$. The displayed positive discrete
sensitivity singular values are for the actual averaged, filtered
measurement with the fixed normalization in \eqref{eq:numerical-cost}.
The step-halving sensitivity check is below
$8.2\times10^{-9}$. These checks do not certify a lower
bound for the continuous sensitivity operator on a smooth domain.

\subsection{Interpretation}
The one-dimensional density errors increase with damping, whereas the
two-dimensional parameter estimates show no monotone amplitude drift.
The prescribed relative noise, repeated edge measurements, temporal
sampling, and fixed two-parameter model all affect the reconstruction
errors. Signal attenuation alone does not imply a downward amplitude
bias under relative-noise normalization. Equation~\eqref{eq:relative-profile}
relates amplitude error directly to initial-state error within the fixed
profile class.

The singular values and repeated-noise results describe the conditioning
of these specific experiments. The square-plate results do not verify
the continuous rank condition for smooth spatially varying densities
in Theorem~\ref{th1}. That condition requires a separate proof or a
certified approximation as described in Section~\ref{sec4}.

\section{Conclusion}\label{sec6}
The natural energy and the generator graph norm have distinct roles in
the damped plate equation. With the specified smooth finite-dimensional
profiles, the graph-norm estimates establish strong differentiability
of the boundary observation map. An injective joint sensitivity then
yields local Lipschitz stability. For homogeneous density and a fixed
nonzero initial mode, identifiability is proved directly and gives a
global estimate on compact parameter rectangles.

The stability constants are uniform on the stated bounded damping
intervals.
Independent noise experiments support reconstruction in the explicit
one-mode class. The additional square-plate test uses both averaged
boundary channels, a specified sensor bandwidth, distinct forward and
inverse resolutions, and repeated noise realizations. Its refinement
and sensitivity checks are numerical evidence for that measurement
protocol. Establishing a continuous rank certificate for nonconstant
smooth density profiles remains an open task for this framework.

\section*{Acknowledgements}
\noindent \textit{AI disclosure:} During preparation, we used DeepSeek-V3 for English polishing and reference checking; all derivations and proofs are solely the authors' own, and we take full responsibility.

\end{document}